\documentclass{llncs}
\usepackage{amsmath,amssymb,graphicx,enumerate,bbm,theorem}

\newcommand{\assign}{:=}
\newcommand{\longdownarrow}{\downarrow}
\newcommand{\noplus}{}
\newcommand{\nosymbol}{}
\newcommand{\tmmathbf}[1]{\ensuremath{\boldsymbol{#1}}}
\newcommand{\tmop}[1]{\ensuremath{\operatorname{#1}}}
\newcommand{\tmstrong}[1]{\textbf{#1}}
\newenvironment{enumerateroman}{\begin{enumerate}[i.] }{\end{enumerate}}
\newtheorem{notation}{Notation}

\begin{document}

\title{On the geometry and the deformation of shape represented by a piecewise
continuous B\'ezier curve with application to shape
optimization}
\author{Olivier Ruatta}
\institute{XLIM - DMI UMR CNRS 7252 Universit\'e de Limoges - CNRS\\
\email{olivier.ruatta@maths.cnrs.fr}}

\maketitle

\begin{abstract}
  In this work, we develop a framework based on piecewize B\'ezier curves to
  plane shapes deformation and we apply it to shape optimization problems. We
  describe a general setting and some general result to reduce the study of a
  shape optimization problem to a finite dimensional problem of integration of
  a special type of vector field. We show a practical problem where this
  approach leads to efficient algorithms. 
\end{abstract}

In all the text below, $E =\mathbbm{R}^2$. In this text, we will define a set
of \ manifolds, each point of such a manifold is a parametrized curves in $E$.

\section{B\'ezier curves}

B\'ezier curves are usual objects in Computer Aided Geometric Design (CAGD)
and have natural and straightforward generalization for surfaces and higher
dimension geometrical objects. We focus here on curves even if a lot of
results have natural generalization in higher dimension. This section has aim
to fix notation and make the paper as self contained as possible.

\subsection{Basic definitions}

Given \ $P_0, P_1, \ldots, P_D \in E$, we define:
\[ B \left( \left(\begin{array}{c}
     P_0, \ldots, P_D
   \end{array}\right), t \right) = \left( 1 - t \right) B \left(
   \left(\begin{array}{c}
     P_0, \ldots, P_{D - 1}
   \end{array}\right), t \right) \noplus + t B \left( \left(\begin{array}{c}
     P_1, \ldots, P_D
   \end{array}\right), t \right) \]
with $B \left( \left( P \right), t \right) = P$ for every $P \in E$. The
associated B\'ezier curve is \\$\left\{ B \left( \left(\begin{array}{c}
  P_0, \ldots, P_D
\end{array}\right), t \right) \left|  \right. t \in 0, 1 \right\}$ and the
list $\left(\begin{array}{c}
  P_0, \ldots, P_D
\end{array}\right)$ is called the control polygon and the points $P_0, \ldots,
P_D$ are called the control points.

This process associates to every set of points a parametrized curve. It is a
polynomial parametrized curve and its degree is bounded :

\begin{proposition}
  Let $P_0, \ldots, P_D \in E$, then $B \left( \left(\begin{array}{c}
    P_0, \ldots, P_D
  \end{array}\right), t \right)$ is a polynomial parametrization and its
  coordinates have degree at most $D$. 
\end{proposition}

\subsection{Bernstein's polynomials \ }

\begin{definition}
  Let $D$ be an integer and $i \in \left\{ 0, \ldots, D \right\}$, we define
  the Bernstein polynomial $b_{i, D} \left( t \right) \assign \binom{i}{D} 
  \left( 1 - t \right)^{D - i} t^i$.
\end{definition}

\begin{notation}
  We denote $\mathbbm{R} \left[ t \right]_D$ the set of polynomial of degree
  less or equal to $D$. The set $\mathbbm{R} \left[ t \right]_D$ has a natural
  $\mathbbm{R}$-vector space structure, its dimension is $D + 1$ and $\left\{
  1, t, \ldots, t^D \right\}$ is a basis of this vector space. 
\end{notation}

\begin{proposition}
  The set $\left\{ b_{0, D}, \ldots, b_{D, D} \right\}$ is a basis of
  $\mathbbm{R} \left[ t \right]_N$.
\end{proposition}

\begin{proposition}
  Let $P_0, \ldots, P_D \in E$, then $B \left( \left(\begin{array}{c}
    P_0, \ldots, P_D
  \end{array}\right), t \right) = \underset{i = 0}{\overset{N}{\sum}} P_i
  b_{i, D} \left( t \right)$ for all $t \in \left[ 0, 1 \right]$. 
\end{proposition}

\begin{corollary}
  Every polynomially parametrized curve can be represented as a B\'ezier
  curve. 
\end{corollary}

\subsection{Interpolation}

Since a B\'ezier curve of degree $D$ is defined using $D + 1$ control points,
one can hope to associate $D + 1$ control points from a sampling of $D + 1$
points on a curve. The following result shows that this is possible. But in
fact, we do not have one B\'ezier curve of degree $D$ but many ones. Each such
curve is associated to a particular sampling of the parameter interval $\left[
0, 1 \right]$.

\begin{proposition}
  Let $M_0, \ldots, M_D \in E$, then there exists B\'ezier curves of degree
  $D$ passing through these points. 
\end{proposition}

\begin{lemma}
  \label{interpol}Let $t_0 = 0 < t_1 < \cdots < t_D = 1$, then there exists
  one and only one B\'ezier curve $B \left( \left(\begin{array}{c}
    P_0, \ldots, P_D
  \end{array}\right), t \right)$ of degree $D$ such that $B \left(
  \left(\begin{array}{c}
    P_0, \ldots, P_D
  \end{array}\right), t_i \right) = M_i, \forall i \in \left\{ 0, \ldots, D
  \right\}$.
\end{lemma}

\begin{proof}
  Denote $M$ the $2 \times \left( D + 1 \right)$ matrix built with the
  coordinate of $M_i$ as $i^{\tmop{th}}$ row, i.e. $ M = \left( M_0, \ldots , M_D \right)^t$,
  and denote $P$ the $2 \times \left( D + 1 \right)$ matrix built with the
  coordinate of $P_i$ as $i^{\tmop{th}}$ row, i.e. $P = \left( P_0, \ldots, P_D\right)^t$.
  We consider the following matrix associated to $\tmmathbf{t}=
  \left(\begin{array}{c}
    t_0, \ldots, t_D
  \end{array}\right)$ :
  \begin{equation}
    B_{\tmmathbf{t}, D} = \left(\begin{array}{cccc}
      b_{0, D} \left( 0 \right) & b_{1, D} \left( 0 \right) & \cdots & b_{D,
      D} \left( 0 \right)\\
      b_{0, D} \left( t_1 \right) & b_{1, D} \left( t_1 \right) & \cdots &
      b_{D, D} \left( t_1 \right)\\
      \vdots & \vdots & \ddots & \vdots\\
      b_{0, D} \left( 1 \right) & b_{1, D} \left( 1 \right) & \cdots & b_{D,
      D} \left( 1 \right)
    \end{array}\right) \label{vander} .
  \end{equation}
  The matrix of equation \ref{vander} is invertible (it is the Vandermonde
  matrix express in the Bernstein basis) and clearly if $P$ is such that $B_{\tmmathbf{t}, D} P = M$, then $B \left( \left[ P_0, \ldots, P_D \right], t \right)$ give the wanted
  curve for the proof of the lemma.
\end{proof}

Remark that once $\tmmathbf{t}$ is know one can compute $B_{\tmmathbf{t},
D}^{- 1}$ once for all and that is possible to take advantage of its
Vandermonde-like structure in order to improve the cost of the multiplication
of a vector by $B_{\tmmathbf{t}, D}$. Generally, we use a regular subdivision
($t_i = \frac{i}{D}$) but there are more suitable choices in regard of the
stability of the computation. 
\section{Piecewize B\'ezier curves}

\subsection{Basics on piecewize B\'ezier curves}

Let $P_{0, 0}, \ldots, P_{0, N}$ and $P_{1, 0}, \ldots, P_{1, D} \in E$ such
that $P_{0, N} = P_{1, 0}$, we define the following parametrization of a
curve:
\[ \Gamma : \left\{ \begin{array}{l}
     \left[ 0, 1 \right] \rightarrow E\\
     t \mapsto \left\{ \begin{array}{l}
       B \left( \left(\begin{array}{c}
         P_{0, 0}, \ldots, P_{0, D}
       \end{array}\right), 2 t \right)  \text{\tmop{for}} t \in \left[ 0, 1 /
       2 \right]\\
       B \left( \left(\begin{array}{c}
         P_{1, 0}, \ldots, P_{1, D}
       \end{array}\right), 2 t - 1 \right)  \text{\tmop{for}} t \in \left[ 1 /
       2, 1 \right]
     \end{array} \right.
   \end{array} \right. . \]
When $N = D$ we say that this parametrization is uniform with respect to the
degree and often, we simply say uniform when it does not introduce ambiguity.
The curves parametrized by $B \left( \left(\begin{array}{c}
  P_{0, 0}, \ldots, P_{0, D}
\end{array}\right), t \right)$ and $B \left( \left(\begin{array}{c}
  P_{1, 0}, \ldots, P_{1, D}
\end{array}\right), t \right)$ are called the patches of $\mathcal{C}= \Gamma
\left( \left[ 0, 1 \right] \right)$. The set of control points of the patches
of $\mathcal{C}$ are called the control points of $\mathcal{C}$. This is a
continuous curve.

More generally, if $P_{0, 0}, \ldots, P_{0, N_1}, P_{1, 0}, \ldots, P_{1,
N_2}, \ldots, P_{l, 0}, \ldots, P_{l, N_l} \in E$, such that $P_{i, N_i} =
P_{i + 1, 0}$ for all $i \in \left\{ 0, \ldots, l - 1 \right\}$, we define:
\begin{equation}
  \Gamma \left( \left(\left( P_{0,0},\ldots,P_{0,D}\right),\ldots, \left( P_{N,0},\ldots,P_{N,D}\right)\right) , t \right) = B \left( \left(P_{i, 0}, \ldots, P_{i, D_i} \right), \frac{i}{l + 1} + \left( l + 1 \right) t \right)
\end{equation}
 for  $t \in \left[ \frac{i}{\left( l + 1 \right)}, \frac{\left(
  i + 1 \right)}{\left( l + 1 \right)} \right]$  and for all $i \in \left\{ 0, \ldots, l \right\}$ .
This defines a continuous parametrization. The curves parametrized by $B
\left( \left(\begin{array}{c}
  P_{i, 0}, \ldots, P_{i, D_i}
\end{array}\right), t \right)$ are called the patches of $\mathcal{C}= \Gamma
\left( \left[ 0, 1 \right] \right)$. Furthermore, if $P_{l, D_l} = P_{0, 0}$
we say that the curve $\mathcal{C}$ is closed or that it is a loop.

We denote $\mathcal{B}_{N, D}$ the set of uniform piecewize B\'ezier curves
built from $N$ patches of degree $D$. This clearly a finite dimensional
subvariety of $\mathcal{C}^0 \left( \left[ 0, 1 \right], E \right)$ as the
image of the following map:
\begin{equation}
  \Psi_{N, D} : \left\{ \begin{array}{l}
    \left( E^{D + 1} \right)^{N + 1} \longrightarrow \mathcal{C}^0 \left(
    \left[ 0, 1 \right], E \right)\\
    \left( \left( P_{i,j},  j=0 \ldots D\right), i=0 \ldots N \right)  \longmapsto \Gamma \left( \left( \left( P_{i,j},  j=0 \ldots D\right), i=0 \ldots N \right) , t \right)
  \end{array} \right.
\end{equation}

Clearly, $\Psi_{N, D}$ is onto from $\left( E^{D + 1} \right)^{N + 1}$ to
$\mathcal{B}_{N, D} \subset \mathcal{C}^0 \left( \left[ 0, 1 \right], E
\right)$. It not very difficult to check that $\Psi_{N, D}$ is almost always
one-to-one from $\left( E^{D + 1} \right)^{N + 1}$ to $\mathcal{B}_{N, D}$.
So, $\Psi_{N, D}$ is almost everywhere a diffeomorphism between $\left( E^{D +
1} \right)^{N + 1}$ and $\mathcal{B}_{N, D}$. This embed $\mathcal{B}_{N, D}$
with a manifold structure (even a submanifold structure in $\mathcal{C}^0
\left( \left[ 0, 1 \right], E \right)$).

The density of polynomials in the set of continuous functions imply that for
each $\Phi : \left[ 0, 1 \right] \longrightarrow E$ there exists $\left(
\Gamma_n \left( t \right) \right)_{n \in \mathbbm{N}}$ such that $\underset{n
\rightarrow \infty}{\lim} \left\| \Phi - \Gamma_n \right\|_2 = 0$, in a way
that considering B\'ezier curves is not a drastic restriction.

\subsection{Sampling map and retraction to $\Psi_{N, D}$}

\begin{definition}
  Let $t_0 = 0 < t_1 < \ldots < t_D = 1$, we denote $\tmmathbf{t}= \left( t_0,
  \ldots, t_D \right)$ the associated subdivision of $\left[ 0, 1 \right]$,
  then we define the sampling map $\mathcal{S}_{\tmmathbf{t}} :
  \mathcal{B}_{1, D} \longrightarrow E^{D + 1}$ by $\mathcal{S}_{\tmmathbf{t}}
  \left( \Gamma \right) = \left(\begin{array}{c}
    \Gamma \left( t_0 \right), \ldots, \Gamma \left( t_D \right)
  \end{array}\right)$.
\end{definition}

\begin{proposition}
  \label{commdiag}The following diagram is commutative:
  \begin{equation}
    \begin{array}{ccc}
      E^{D + 1} & \begin{array}{c}
        \Psi_{1, D}\\
        \longrightarrow
      \end{array} & \mathcal{C}^0 \left( \left[ 0, 1 \right], E \right)\\
      & \begin{array}{c}
        \underset{B_{\tmmathbf{t}, D}}{\searrow}
      \end{array} & \begin{array}{cc}
        \longdownarrow & \mathcal{S}_{\tmmathbf{t}}
      \end{array}\\
      &  & E^{D + 1}
    \end{array}
  \end{equation}
  and $\Psi_{N, D}$ is an invertible linear isomorphism between
  $\mathcal{B}_{1, D} = \tmop{Im} \left( \Psi_{1, D} \right)$ and $E^{D + 1}$
  and its inverse is $\Psi_{N, D}^{- 1} = B_{\tmmathbf{t}, D}^{- 1} \circ
  \mathcal{S}_{\tmmathbf{t}}$. 
\end{proposition}

\begin{proof}
  Let $\Gamma \left( t \right) = \underset{j = 1}{\overset{N}{\sum}}
  \underset{i = 0}{\overset{D}{\sum}} P_{j, i} b_{i, D} \left( t \right)$ $i.e
  \nosymbol .$ $\Gamma = \Psi_{N, D} \left( \left( P_{0, 0}, \ldots, P_{0, D}
  \right), \ldots, \left( P_{N, 0}, \ldots, P_{N, D} \right) \right)$, then
  clearly $\mathcal{S}_{\tmmathbf{t}} \left( \Gamma \right) = B_{\tmmathbf{t},
  D} \tmmathbf{P}$ where $\tmmathbf{P}= \left(\begin{array}{c}
    P_{0, 0}^t\\
    \vdots\\
    P_{N, D}^t
  \end{array}\right)$, and so $\mathcal{S}_{\tmmathbf{t}} \circ \Psi_{N, D}
  \left( \tmmathbf{P} \right) = B_{\tmmathbf{t}, N} \left( \tmmathbf{P}
  \right)$. The remainder of the theorem is a consequence of the fact that
  $B_{\tmmathbf{t}, N}$ is a linear isomorphism. 
\end{proof}

\begin{proposition}
  \label{samplemain}Let $t_{1, 0} = 0 < t_{1, 1} < \cdots < t_{1, D} = 1 / N =
  t_{2, 0} < t_{2, 1} < \cdots < t_{2, D} = 2 / N = t_{3, 0} < \cdots < t_{N,
  D} = 1$, we denote $\tmmathbf{t}= \left( \tmmathbf{t}_1, \ldots,
  \tmmathbf{t}_N \right)$ where $\tmmathbf{t}_i = \left( t_{0, i}, \ldots,
  t_{D, i} \right)$ and we define the sampling map $\mathcal{S}_{\tmmathbf{t},
  N} : \mathcal{B}_{N, D} \longrightarrow \left( E^D \right)^N$ by
  $\mathcal{S}_{\tmmathbf{t}, N} \left( \Gamma \right)
  =\mathcal{S}_{\tmmathbf{t}_1} \times \cdots \times
  \mathcal{S}_{\tmmathbf{t}_N} \left( \Gamma \right) = \left(\begin{array}{c}
    \Gamma \left( t_{1, 0} \right), \ldots, \Gamma \left( t_{N, D} \right)
  \end{array}\right)$. Then $\mathcal{S}_{\tmmathbf{t}, N}$ is a linear
  isomorphism between $\mathcal{B}_{N, D}$ and $\left( E^{D + 1} \right)^{N +
  1}$. 
\end{proposition}

\begin{proof}
  It is a simple consequence of the fact that a cartesian product of
  isomophisms is an isomorphisms. The inverse map is the cartesian product of
  the inverse of the component maps. 
\end{proof}

The proposition \ref{samplemain} is important since it allows to give to
$\mathcal{B}_{N, D}$ a vector space structure isomorphic to $\left( E^{D + 1}
\right)^{N + 1}$ (and so, of finite dimension). For instance, it allows to
transport distance and so on in $\mathcal{B}_{N, D}$.

In fact, we focus here into a speciale type of sampling. We consider such an
sampling where $t_{i, 0} = \frac{i}{N}$ and $t_{i, D} = \frac{i + 1}{N}$ and
$t_{i, j} = t_{i, 0} + \frac{j}{ND}$. We will call this kind of sampling a
regular sampling and we will omit the subscript $\tmmathbf{t}$ everywhere
using these samplings. We use these sampling to simply the presentation, but
all the results has equivalent statements with general sampling. Representing
each patch by its control polygon, the matrix of $\mathcal{S}_{\tmmathbf{t},
N}$ is $N$ times the cartesian production of the map $B_{1, D}$ with itself:
$B_{1, D} \times \cdots \times B_{1, D}$. This gives us an easy way to solve
the following interpolation problem.

\begin{problem}
  \label{patchinter}Given $M_{0, 0}, \ldots, M_{0, D}, \ldots, M_{N, 0},
  \ldots, M_{N, D} \in E$, find $\Gamma \in \mathcal{B}_{N, D}$ such that
  $\mathcal{S}_N \left( \Gamma \right) = \left(\begin{array}{c}
    M_{0, 0}^t\\
    \vdots\\
    M_{N, D}^t
  \end{array}\right)$.
\end{problem}

\begin{proposition}
  \label{patchinterpol}The solution of problem \ref{patchinter} is given by
  the image by $\Psi_{N, D}$ of:
  \begin{equation}
    \left(\begin{array}{ccc}
      B_{1, D}^{- 1} &  & \\
      & \ddots & \\
      &  & B_{1, D}^{- 1}
    \end{array}\right) \left(\begin{array}{c}
      M_{0, 0}^t\\
      \vdots\\
      M_{N, D}^t
    \end{array}\right) .
  \end{equation}
\end{proposition}

The proposition \ref{patchinterpol} implies that $\chi_{\tmmathbf{t}, D} =
B_{t, D}^{- 1} \circ \mathcal{S}_{\tmmathbf{t}} : \mathcal{B}_{N, D}
\longrightarrow E^{D + 1}$ is such that $\Psi_{1, D} \circ \chi_{\tmmathbf{t},
D} = \tmop{Id}_{E^{D + 1}}$. It is easy to extend this result to $\Psi_{N, D}$
using $B_{N, D} = B_{1, D} \times \cdots \times B_{1, D}$ satisfying $B_{N,
D}^{- 1} = B_{1, D}^{- 1} \times \cdots \times B_{1, D}^{- 1}$. \

This approach allows to project any element of $\mathcal{C}^0 \left( \left[
0, 1 \right], E \right)$ on $\mathcal{B}_{N, D}$ using $\mathcal{S}_N$. Let
$\Lambda \in \mathcal{C}^0 \left( \left[ 0, 1 \right], E \right)$, then
denoting $\tmmathbf{M}= \left( \Gamma(0), \Gamma(\frac{1}{N D}), \ldots, \Gamma(\frac{N D -1}{ND}), \Gamma(1)\right)^t$ we have $\tmmathbf{P}=\mathcal{S}_N^{- 1} \left(
\tmmathbf{M} \right) \in \mathcal{B}_{N, D}$ is such that $\Psi_{N, D} \left(
\tmmathbf{P} \right) = B \left( \tmmathbf{P}, t \right)$ coincides with
$\Lambda \left( \left[ 0, 1 \right] \right)$ on at least $(D+1)$ points counted
with multiplicities on each patch. This is only the fact that $\chi_{\tmmathbf{t}, D}$ can
be extend to $\mathcal{C}^0 \left( \left[ 0, 1 \right], E \right)$. \

The main claim is that instead of working directly with $\mathcal{B}_{N, D}$,
it is easier to work on the ``set of control polygons'', namely $E^{D + 1}$
using sampling and interpolation giving linear isomorphism between control
polygons and sampling points on the curves. In what follows, we will always
take this point of view.

\subsection{Tangent space $T\mathcal{B}_{N, D}$ and deformation of curve}

\ \ Recall that $\Psi_{N, D}$ define a linear isomorphism between the ``space
of control polygons'' $\left( E^{D + 1} \right)^{N + 1}$ and the space of
piecewize B\'ezier curves $\mathcal{B}_{N, D}$. We already saw that for any
$\gamma \left( t \right) \in \mathcal{B}_{N, D}$ then $\tmmathbf{P} \in \left(
E^{D + 1} \right)^{N + 1}$ such that $\Psi_{N, D} \left( \tmmathbf{P} \right)
= \gamma \left( t \right)$ is given by $B_{N, D}^{- 1} \circ \mathcal{S}_N
\left( \gamma \right)$. This give the following proposition:

\begin{proposition}
  We have that $T \Psi_{N, D} : T \left( E^{D + 1} \right)^{N + 1}
  \longrightarrow T\mathcal{B}_{N, D}$ is such that from any $\gamma \in
  \mathcal{B}_{N, D}$ we have $T \Psi_{N, D}^{- 1} \left( \gamma \right) :
  T_{\gamma} \mathcal{B}_{N, D} \longrightarrow T_{\chi_{N, D} \left( \gamma
  \right)} \left( E^{D + 1} \right)^{N + 1}$ is given by $T \Psi_{N, D} \left(
  \chi_{N, D} \left( \gamma \right) \right)^{- 1} \left( \varepsilon \right) =
  B_{N, D}^{- 1} \circ \mathcal{S}_N \left( \varepsilon \right) = \chi_{N, D}
  \left( \varepsilon \right)$ for any $\varepsilon \left( t \right) \in
  T_{\gamma} \mathcal{B}_{N, D}$ and this is a linear isomorphism. 
\end{proposition}

An element of $\varepsilon \left( t \right) \in T_{\gamma} \mathcal{B}_{N, D}$
is called a deformation curve. In fact, this proposition allows to express,
given a piecewize B\'ezier curve and a deformation, how to deform its control
polygon. This is an essential step proving that manipulating piecewize
B\'ezier curve, it is enough to manipulate its control polygon. This is the
object of the following lemma.

\begin{lemma}
  Let $\tmmathbf{P} \in \left( E^{D + 1} \right)^{N + 1}$, $\gamma \left( t
  \right) = \Psi_{N, D} \left( \tmmathbf{P} \right) = B \left( \tmmathbf{P}, t
  \right) \in \mathcal{B}_{N, D}$ and $\varepsilon \left( t \right) \in
  T_{\gamma} \mathcal{B}_{N, D}$, then:
  \begin{enumerateroman}
    \item $\varepsilon \left( t \right) = \Psi_{N, D} \left( \chi_{N, D}
    \left( \varepsilon \right) \right)$.
    
    \item $\gamma \left( t \right) + \varepsilon \left( t \right) = \Psi_{N,
    D} \left( \tmmathbf{P}+ \chi_{N, D} \left( \varepsilon \right) \right)$.
  \end{enumerateroman}
\end{lemma}

This lemma explain how to lift a deformation from the space of curves to the
space of control polygons. The vector space structure of both the space of
control polygons $\left( E^{D + 1} \right)^{N + 1}$ and of piecewize B\'ezier
curves $\mathcal{B}_{N, D}$ allows to avoid the use of computationally
difficult concept as exponential map between manifold and its tangent space
and so on. This structure has also to define a simple notion of distance
between two such curves.

\section{Applications to shape optimization}

In this section, we show how the preceding formalism can be exploited in the
context of shape optimization. An application to a problem of image
segmentation is presented to illustrate our purpose.

\subsection{Shape optimization problem}

A shape optimization problem consists in, given a set of admissible shapes
$\mathcal{A}$ and a functional $F : \mathcal{A} \rightarrow \mathbbm{R}^+$
find a shape $\alpha \in \mathcal{A}$ such that for all other shapes $\beta
\in \mathcal{A}$, we have $F \left( \alpha \right) \leqslant F \left( \beta
\right)$. Generally, one try to give to the space of admissible shape a
structure of manifold in a way to be able to compute a ``shape gradient'' ,
$\nabla F \left( \beta \right)$, expressing the evolution of the criterium $F$
with respect to a deformation of the shape $\beta$. It is to say that $\nabla
F \left( \beta \right)$ associates to every point $M \in \beta$ a deformation
vector $\nabla F \left( \beta \right) \left( M \right) \in T_M E$. The
computation of such a gradient can require sophisticated computation since
very often, even the computation of the criterium itself require to solve a
system partial differential equations. Many problem can be expressed as a
shape optimization problem. Classical approach to solve this kind of problem
is to use $\nabla F \left( \beta \right)$, when it is computable, in a
gradient method to find a local minimum.

To keep the presentation as simple as possible, we focus on geometric
optimisation, i.e. the topology of the shape is fixed, in the case where the
frontier of the admissible shapes are continuous Jordan curves. But the
framework presented here can be extended to topological optimization as it is
shown in \cite{LY10} for a special application on a problem of image segmentation. The
case treated here received attention because of its deep links with images
segmentation and shape recognition (see \cite{LY10,YMSM08,YM05} for instance).

We denote $\mathcal{C}_J^0 \left( \left[ 0, 1 \right], E \right)$ the set of
function parametrizing a Jordan curve and $\mathcal{B}^c_{N, D} = \left\{
\gamma \in \mathcal{B}_{N, D}  \left| \right. \gamma \left( t \right) = \gamma
\left( s \right) \text{\tmop{with}} s \neq t \Leftrightarrow \left( \left. t =
0 \right. \text{\tmop{and}} s = 1 \right)  \text{\tmop{or}} \left( t = 1
\text{\tmop{and}} s = 0 \right) \right\}$. We have $\mathcal{B}_{N, D}^c
\subset \mathcal{C}_J^0 \left( \left[ 0, 1 \right], E \right)$. We denote:
\[H_{N, D} = \left\{ \left( \left( P_{i,j},
j = 0 \ldots D \right), i = 0 \ldots N \right) \in \left( E^D \right)^N \left| \right. P_{0, 0} = P_{N, D} \right\}\]
which is a linear subspace of $\left( E^{D+1} \right)^{N+1}$. We then denote $\Psi_{N,
D}^c = \Psi_{N, D} \left|_{H_{N, D}} \right.$. As above, $\Psi_{N, D}^c$
define a linear isomorphism between $H_{N, D}$ and $\mathcal{B}_{N, D}^c$
using $\mathcal{S}_N \left|_{\mathcal{B}_{N, D}^c} \right.$ and the same
$B_{N, D}$ to define its converse explicitly.

\subsection{Vector field on $\mathcal{B}_{N, D}$ lifted from the shape
gradient}

Let $\nabla F$ be a shape, then for each $\alpha \in \mathcal{C}^0_J \left(
\left[ 0, 1 \right], E \right)$ and for any $M \in \alpha \left( \left[ 0, 1
\right] \right)$, \ $\nabla F$ associate to $M$ an element $\nabla F \left(
\alpha \right) \left( M \right) \in T_M E$. Consider now $\alpha \in
\mathcal{B}_{N, D}^c$ and $\left( \left( M_{0, 0}, \ldots, M_{0, D} \right),
\ldots, \left( M_{N, 0}, \ldots, M_{N, D} \right) \right) =\mathcal{S}_N
\left( \alpha \right)$, then $M_{0, 0} = M_{N, D}$. We 
\[ \mathcal{T}_{N,F} \left( \alpha \right) = \left( \left( \nabla F \left( \alpha \right) \left(
M_{i,j} \right),j=0 \ldots D \right), i=0 \ldots D  \right) \]
This representes the sampling of the deformation of the curve
implied by the shape gradient $\nabla F$ to $\alpha$. It is not difficult to
see that $\mathcal{T}_{N, F} \left( \alpha \right) \in T_{\mathcal{S}_N \left(
\alpha \right)} \left( \left( E^D \right)^N \right)$. 

\begin{theorem}
  To each shape gradient $\nabla F$ the map $B_{N, D}^{- 1} \circ
  \mathcal{T}_{N, F}$ associate a vector field on $H_{N, D}$ which correspond
  to a vector field $V_F$ on $\mathcal{B}_{N, D}^c$ through $T \Psi_{N, D}^c$.
  
\end{theorem}

This theorem allows to interpret gradient descent method for shape
optimization as a algorithm for integrating a vector field in a finite
dimensional space. From this point of view, gradient descent method correspond
to the most naive method to integrate this vector field, namely the Euler
method. Clearly, this approach suggests to use better algorithm for vector
field integration.

\subsection{Geometry of the vector field and local extrema of shape cost
functional}

\begin{proposition}
  Let $\alpha \in \mathcal{C}^0_J \left( \left[ 0, 1 \right], E \right)$ be
  such that $\nabla F \left( \alpha \right) =\tmmathbf{0}$, i.e. $\nabla F
  \left( \alpha \right) \left( M \right) =\tmmathbf{0}$ for all $M \in \alpha
  \left( \left[ 0, 1 \right] \right)$ and let $\gamma \in \mathcal{B}_{N, D}$
  such that $\gamma \left( \frac{i}{N D} \right) = \alpha \left( \frac{i}{N D}
  \right)$ for $i \in \left\{ 0, \ldots, N D \right\}$, i.e. $\gamma =
  \Psi_{N, D} \left( B_{N, D}^{- 1} \circ \mathcal{S}_N \left( \alpha \right)
  \right)$, then $\nabla F \left( \gamma \right) \left( \gamma \left(
  \frac{i}{N D} \right) \right) = 0$ for $i \in \left\{ 0, \ldots, N D
  \right\}$ and then $V_F \left( B_{N, D}^{- 1} \circ \mathcal{S}_N \left(
  \alpha \right) \right) =\tmmathbf{0}$. It is to say that a local extremum of
  $F$ induces a local extremum of its restriction to $\mathcal{B}_{N, D}$ and
  that this extremum is ``lifted'' on a singularity of the vector field $V_F$
  on $H_{N, D}$.
\end{proposition}

\begin{proof}
  The deformation curve of $\gamma$ induces by the gradient of $F$ vanishes at
  at least $\left( D + 1 \right) \left( N + 1 \right)$ points, but it is a
  ``B\'ezier curve'' of degree $D$, so it a zero polynomial. So, its control
  polygon is reduce to the origin and then $B_{N, D}^{- 1} \circ \mathcal{S}_N
  \left( \alpha \right)$ is a singularity of $V_F$. 
\end{proof}

In fact, the vector field $V_F$ is associated to the gradient of the function
$F \circ \Psi_{N, D}$. This is an heavy constrain on the vector field. For
instance, it is easy easy to see that $\tmmathbf{P}$ is an attractive
singularity of $V_F$ if and only if $\Psi_{N, D} \left( \tmmathbf{P} \right)$
is a local minimum of $F \left|_{H_{N, D}} \right.$.

\subsection{Application to a problem of images segmentation}

In this section, we sketch an application to a problem of images
segmentation. It is a problem of omnidirectional vision. Previous methods
tried with some success but does not allowed a full real time treatment. There
all based on snake-like algorithms (see {\cite{KWT88}}). The gradient use to
detect edges is a classical one based on a Canny filter and \ is combined with
a balloon force. The best previously known method is such that propagation of
the contour were done using the fast marching algorithm for level set method.
This a typical formulation of image segmentation as a shape optimization
problem. In {\cite{LMR13}}, we use piecewize B\'ezier curves to contour
propagation and achieve a very fast segmentation algorithm allowing real time
treatment even with sequential algorithm (no use of parallelism or special
hardware architecture) on a embedded system.\\

\begin{tabular}{cc}
\includegraphics[width=4cm]{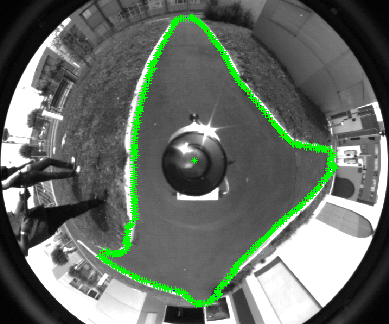} 
& \includegraphics[width=4cm]{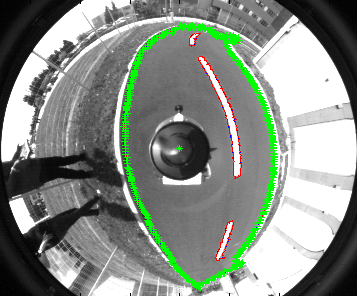}  \\
\end{tabular} \\

It is very interesting to see that, with few algorithmic modification, it is
also possible to treat change of topology, i.e. curves with several connected
components as it is shown in the following figure.


\begin{thebibliography}{}
  \bibitem{MM07} P. W. Michor and D. Mumford. {\tmstrong{An Overview of the
  Riemannian metric on Spaces of Curves using the Hamiltonian Approach}}.
  Applied and Computational Harmonic Analysis, {\tmstrong{23}} (2007), p.
  74-113.
  
  \bibitem{YM05} A. Yezzi and A. C. G. Mennucci. {\tmstrong{Metric in space
  of curves}}. http://arxiv.org/abs/math/0412454.
  
  \bibitem{YMSM08} L. Younes and P. W. Michor and J. Shah and D. Mumford.
  {\tmstrong{A metric on shape space with explicit geodesics}}. Rend. Lincei.
  Mat. Appl., {\tmstrong{9}} (2008), p. 25-57.
  
  \bibitem{BHM} M. Bauer and P. Harms and P. W. Michor. {\tmstrong{Curvature
  weighted metrics on shape space of hypersurfaces in $n$-space}}.
  Differential Geometry and Applications, 2011.
  
  \bibitem{LY10} L. Younes. {\tmstrong{Shapes and Diffeomorphisms}}. Applied
  Mathemetical Sciences 171, Springer (2010).
  
  \bibitem{LMR13} O. Labbani-I. and P. Merveilleux-O. and O. Ruatta.
  {\tmstrong{Free form based active contours for image segmentation and free
  space perception}}, preprint (submitted available to referee at
  http://www.unilim.fr/pages\_perso/olivier.ruatta/TRO-paper-submitted.pdf).
  
  \bibitem{HP05} A. Henrot and M. Pierre. {\tmstrong{Variation et
  optimisation de formes}}. Math\'ematiques et Applications, Springer-Verlag
  Berlin Heidelberg, 2005.
  
  \bibitem{JS99} J.A. Sethian. {\tmstrong{Level Set Methods and Fast Marching
  Methods Evolving Interfaces in Computational Geometry, Fluid Mechanics,
  Computer Vision, and Materials Science}}. Cambridge University Press, 1999
  Cambridge Monograph on Applied and Computational Mathematics.
  
  \bibitem{KWT88} M. Kass, A. Witkin, and D. Terzopoulos. {\tmstrong{Snakes:
  Active contour models}}. International Journal of Computer Vision (1988),
  vol. 1, no. 4, pp. 321-331.
\end{thebibliography}
\end{document}